\documentclass[a4paper,12pt]{amsart}

\usepackage[]{graphicx}
\usepackage[colorlinks=false]{hyperref}
\usepackage[]{amssymb}
\usepackage[paperwidth=210mm,paperheight=297mm,inner=2.8cm,outer=2.8cm,top=2.5cm,bottom=2.6cm]{geometry}
\theoremstyle{plain}
\newtheorem{Theorem}{Thm}[section]
\newtheorem{Thm}[Theorem]{Theorem}
\newtheorem{Lem}[Theorem]{Lemma}
\newtheorem{Cor}[Theorem]{Corollary}

\theoremstyle{definition}

\newtheorem{Exm}[Theorem]{Example}
\newtheorem{Rem}[Theorem]{Remark}

\newcommand\N{\mathbb{N}}

\newcommand\Q{\mathbb{Q}}
\newcommand\R{\mathbb{R}}

\DeclareMathOperator\diag{diag}

\DeclareMathOperator\lspan{span}

\usepackage{paralist}

\begin{document}
\title{Generic Spectrahedral Shadows}
\author{Rainer Sinn}
\author{Bernd Sturmfels}

\begin{abstract}
Spectrahedral shadows are projections of linear sections of
the cone of positive semidefinite matrices. 
We characterize the polynomials that vanish on the boundaries of these convex sets
when both the section and the projection are generic.
\end{abstract}

\maketitle

\section{Introduction}

A {\em spectrahedral shadow} is a convex set $S$ in $\R^d$ that admits a representation
\[
S \,=\, \bigl\{\, (x_1,\ldots,x_d)\in\R^d \,\, \big| \, \, \exists\; (y_1,\ldots,y_p)\in\R^p\, \colon \sum_{i=1}^d x_iA_i + \sum_{j=1}^p y_jB_j +C \,\succeq \,0 \,\bigr\},
\]
where $A_1,\ldots,A_d$, $B_1,\ldots,B_p$ and $C$ are real symmetric $n\times n$ matrices,
and the symbol ``$ \succeq$'' means that the matrix is positive semidefinite.
Such sets are the feasible regions in semidefinite programming \cite{BPT}, and 
understanding their geometry is of considerable interest in optimization theory.
For instance, a longstanding conjecture asserts that every closed convex 
semialgebraic subset of $\R^d$ is a spectrahedral shadow (cf.~\cite{HN}).
This conjecture was recently proved for $2$-dimensional sets by Scheiderer, see \cite{Scheiderer}.

In this article we assume that $A_i$, $B_j$ and $C$ are {\em generic}, {\it i.e.}~the
 tuple of matrices that defines $S$ lies
outside a certain discriminantal hypersurface in the space of $(d+p+1)$-tuples of
symmetric $n \times n$-matrices.  This is the discriminant  whose non-vanishing characterizes
{\em transversal symmetroids} as in \cite[\S 1]{Quartic}.
The symmetroid defined by the vanishing of the polynomial $\det(\sum x_iA_i + \sum y_jB_j + C)$ is {\em transversal} if all points of corank $1$ are smooth on this algebraic hypersurface. A spectrahedron
whose symmetroid is transversal is either empty or full-dimensional.
Therefore, when the matrices $A_i, B_j, C$ are generic,
the spectrahedral shadow $S$ inherits this property from the spectrahedron upstairs.
Namely, $S$ is 
either empty or full-dimensional in $\R^d$.
If these hypotheses hold then we call $S$ a {\em generic
spectrahedral shadow of type $(n,d,p)$}.

Given a linear functional that minimizes in a boundary point $x\in\partial S$, the linear functional pulled back via the projection minimizes along a face of the spectrahedron upstairs.
By {\em rank of $x$} we mean the rank of a matrix $\sum x_iA_i + \sum y_jB_j+C$ in the relative interior of that face. This does not depend on the choice of $y $ because the rank is constant on the relative interior of faces of a spectrahedron; see Ramana-Goldman~\cite{RamGol}.

The boundary $\partial S$ of 
a non-empty generic spectrahedral shadow $S$ is a semi-algebraic set of codimension $1$
in $\R^d$;
see \cite[Corollary 2.6]{Sin}.
Up to scaling, there exists a unique squarefree polynomial $\Phi_S \in \R[x_1,\ldots,x_d]$
that vanishes on $\partial S$. The zero set of $\Phi_S$ is 
denoted $\partial_{\rm alg}(S)$ and is called the {\em algebraic boundary} of $S$.
The irreducible components of $\partial_{\rm alg}(S)$ are the zero sets of the
irreducible factors of $\Phi_S$.
Our main result characterizes the number and degrees of these 
factors in terms of $n,d$ and~$p$.

We shall answer the following questions concerning the boundary structure of $S$:
\begin{compactitem}
\item What ranks occur in the boundary of $S$?
 \item How many irreducible components are there in the algebraic boundary of $S$?
\item What are the degrees of these hypersurfaces?
\end{compactitem}
The following result is our main theorem.

\begin{Thm}\label{Thm:algBound}
Let $S$ be a generic spectrahedral shadow of type $(n,d,p)$.
The rank $r$ of any general point in the boundary $\partial S$ of $S$ lies in the range
that is given by
\begin{equation}
\label{eq:patakirange}
\binom{n-r+1}{2}  \,\leq\,  p+1 \quad {\rm and} \quad
\binom{r+1}{2} \, \leq\,  \binom{n+1}{2}-(p+1).
\end{equation}
The points of rank $r$ form an irreducible algebraic hypersurface.
Its degree is independent of $d$; it is equal to
$\delta(p+1,n,r)$.
The algebraic boundary of $S$ is the union of these irreducible hypersurfaces,
one for each rank $r$ that occurs for a general point in~$\partial S$.
\end{Thm}

\begin{Rem}
Every of the above hypersurfaces given by a specific rank has a strictly positive probability
of occurring. This probability is related to the probability of having an extreme point of rank $r$ on a spectrahedron, computed by
Amelunxen and B\"urgisser \cite[Theorem 3.5]{AB}
for a certain natural distribution on the data $A_i,B_j,C$.
\end{Rem}

The quantity $\delta(m,n,r)$ is the {\em algebraic degree of semidefinite programming},
as defined and computed by Nie-Ranestad-Sturmfels \cite{NieRanStu} and von-Bothmer-Ranestad \cite{BotRan}. 
Setting $p= m-1$ in \eqref{eq:patakirange} gives the well-known 
{\em Pataki range} for the ranks of the optimal matrices in semidefinite programming,
where the feasible set is an $m$-dimensional spectrahedron of $n \times n$-matrices.
Indeed, this is precisely the special case $d =1$ of our Theorem \ref{Thm:algBound}.
Semidefinite programming amounts to projecting a spectrahedron onto the line,
and $\delta(m,n,r)$ is the number of critical points of rank $r$ of that projection. What surprises
about Theorem \ref{Thm:algBound} is that, given $n$, the algebraic structure of $\partial S$
is independent of $d= {\rm dim}(S)$ but only depends on the codimension $p$
of the projection.

\begin{Exm} \rm
Let $n=3$ and $p=1$. Then
(\ref{eq:patakirange}) says
$ \binom{4-r}{2} \leq 2$ and $\binom{r+1}{2} \leq 4 $,
so $r=2$. Theorem \ref{Thm:algBound} states that,
in each of the following situations, the algebraic boundary of 
a generic spectrahedral shadow $S$ in $\R^d$
is irreducible of degree $\delta(2,3,2) = 6$. Note that taking the Zariski closure comes with subtleties, as explained in Remark \ref{Rem:Zariskiclos}
\begin{itemize}
\item If $d=1$ then we project a planar cubic spectrahedron onto the line.
The line segment $S$ is bounded by algebraic numbers of degree $6$
over the ground field.
\item If $d=2$ then $S$ is the image of the familiar elliptope (\ref{eq:elliptope}) under a general projection.
Its boundary $\partial_{\rm alg}(S)$ is a rational curve of degree six in the plane.
\item If $d=3$ then $S$ is a convex body whose boundary is
 a surface of degree~$6$.
\item If $d=4$ then the threefold $\partial_{\rm alg}(S)$ 
is the branch locus of a general projection, from $5$-space to $4$-space,
of the secant variety to the quadratic Veronese surface. 
\end{itemize}
In each of these four situations, the points on the boundary of $S$ have rank $r=2$.
\end{Exm}

The rest of this paper is organized as follows. In Section 2 we present the proof
of Theorem \ref{Thm:algBound}, and we discuss the values that are attained 
by the algebraic degree of SDP.
Section 3 is concerned with case studies of generic spectrahedral
shadows in dimensions $d=2$ and $d=3$.
In Section 4 we examine how spectrahedral shadows 
degenerate when the data $A_i,B_j,C$ move from generic to special.
We believe that this analysis will be useful for studying
extended representations in convex optimization.

\section{One proof and many numbers}

As preparation for the proof of our main result, we establish the following fact.
\begin{Lem}\label{Lem:closed}
Let $Q\subset\R^n$ be a spectrahedron and let $\pi\colon \R^n\to\R^d$ be a generic 
projection. Then $\pi(Q)\subset\R^d$ is a closed convex semialgebraic set.
\end{Lem}

Here, $\pi$ is represented by a $d \times n$-matrix $(\pi_{ij})$
and being ``generic'' is understood in the sense of algebraic geometry,
meaning that there exists a nonzero polynomial $h$ in the 
$dn$ matrix entries $\pi_{ij}$ such that
$\pi(Q)$ is closed whenever $\pi$ satisfies $h(\pi) \not= 0$.

\begin{proof}
That $\pi(Q)$ is convex is clear because $\pi$ is linear.
That $\pi(Q)$ is semialgebraic follows from Tarski's
Quantifier Elimination Theorem.
We need to show $\pi(Q)$ is closed.

We first discuss the case $\dim(\ker(\pi)) = 1$ and then use induction on $\dim(\ker(\pi))$.
After a change of coordinates in $\R^n$,
we assume that $\pi$ is the coordinate projection $\pi\colon \R^{n-1}\oplus\R \to \R^{n-1}$, $(x,y)\mapsto x$.
Let $f$ be the polynomial defining $\partial_a Q$.
After a generic change of coordinates, we may assume that 
$f$ is monic in the variable $y$, i.e.
\[
f(x,y) = y^s + f_1 y^{s-1} + f_2 y^{s-2} +  \cdots + f_{s-1} y +  f_s,
\]
where the $f_i$ are polynomials in $x = (x_1,\ldots,x_{n-1})$.
We will show that $\pi(Q)$ is closed.

Let $(x_k)_{k\in\N}$ be a sequence 
of points in  $\pi(Q)$ that converges in $\R^{n-1}$
and let $x = \lim_{k\to \infty} x_k $.
Each $x_k$ has a preimage $(x_k,y_k)\in Q$ satisfying $f(x_k,y_k) = 0$ or the line $\{(x_k,y)\colon y\in\R\}$ is contained in $Q$. 
In the former case, $y_k$ is a real root of $f(x_k,y)$,
and in the latter case we can arbitrarily set $y_k = 0$.
Then  $\{|y_k|\colon k\in\N\}$ is a bounded subset of $\R$
  because the zeros of a polynomial depend continuously on the coefficients.
We choose a convergent subsequence of $(y_k)_{k\in\N}$ with limit $y$ in $\R$.
Since $Q$ is closed, the limit point $(x,y)$ lies in $ Q$.
Now $\pi(x,y) = x$, and we conclude that $\pi(Q)$ is closed.

There is a slight twist to using the codimension $1$ case above as the induction hypothesis because
$\pi(Q)$ is generally not a spectrahedron. However, it is closed, convex and semialgebraic,
so $\partial_a \pi(Q)$ is also defined by a polynomial $f$, see \cite[Lemma 2.5]{Sin}.

Now let $\pi\colon \R^n\to\R^d$ be a generic projection with $n -d > 1$.
We write $\pi$ as the composition of $n-d$ projections of codimension $1$, say $\pi = \pi_{1}\circ\pi_{2}\circ\cdots\circ \pi_{n-d}$, with $\dim(\ker(\pi_i)) = 1$.
By induction on $n-d$, we may assume that 
$(\pi_{2}\circ\cdots\circ \pi_{n-d})(Q)$ is closed. We now apply the above
argument to this closed convex semialgebraic set with the codimension $1$ projection $\pi_1$.
We conclude that $\pi(Q)$ is closed, as desired.
\end{proof}

We now present the proof of our main result.

\begin{proof}[Proof of Theorem \ref{Thm:algBound}]
Let $S$ be a generic spectrahedral shadow of type $(n,d,p)$
defined by symmetric matrices $A_1,\ldots,A_d$, $B_1,\ldots,B_p$,
and $C$ as above. Let 
\[
S^o = \left\{ (a_1,\ldots,a_d)\in\R^d \; \vert \; \forall \;x\in S\colon \langle x,a\rangle \geq -1 \right\}
\]
denote the convex body dual to the 
shadow $S$ with respect to the
usual inner product on $\R^d$. Up to taking closure, the dual convex body $S^o$ is the
image of the spectrahedron
\begin{equation}
\label{eq:QQQ}
Q \,=\, \bigl\{ \,U\in \mathcal{S}_n^+ \,\colon  \langle B_1,U\rangle=\cdots=\langle B_p,U\rangle =0 \,\,
{\rm and}\,\,
 \langle C,U\rangle =1 \bigr\}
\end{equation}
under the linear map
\[ \pi \,\colon \,\mathcal{S}_n \to \R^d , \,\,U\mapsto (\langle A_1,U\rangle,\ldots,\langle A_d,U\rangle). \]
Here $\mathcal{S}_n$ denotes the vector space of real symmetric matrices and $\mathcal{S}_n^+$ the cone of positive semidefinite $n \times n$-matrices.
The inner product of two symmetric matrices is the usual trace inner product.
 This representation was used in the book chapter by
 Rostalski-Sturmfels \cite[Remark 5.43]{RosStu}
 to prove that the class of spectrahedral shadows is closed under duality.
 In our situation, $\pi(Q)$ is closed by 
  Lemma \ref{Lem:closed}. Indeed, $\pi$ is generic by
 our genericity assumptions on $A_1,\ldots,A_d$.
Hence $S^o = \pi(Q)$ holds exactly.
 
We now apply duality to the identity $S^o = \pi(Q)$.
This expresses our spectrahedral shadow $S = (S^o)^o = \pi(Q)^o$ as 
a linear section of the convex body $Q^o$ that is dual to~$Q$:
\[
S \,\,\cong\,\, Q^o \,\cap\, \lspan(A_1,\ldots, A_d)
\]
via the isomorphism $\R^d \to \lspan(A_1,\ldots,A_d)$, $(x_1,\ldots,x_d)\mapsto \sum x_i A_i$.
The duality between $Q$ and $Q^o$ takes place in $\mathcal{S}_n$ with its usual trace inner product.

We understand the algebraic boundary of $Q^\circ$ by the results in~\cite{NieRanStu}.
Namely, $Q$ is a generic spectrahedron of codimension $p+1$ in the space of symmetric $n\times n$ matrices.
By \cite[Theorem 5.50]{RosStu}, its extreme points are
matrices of rank $n-r$ for some $r$ that lies in the Pataki range \eqref{eq:patakirange}.
We see each one of these ranks as the rank of extreme points with a positive probability. This was computed by Amelunxen and B\"urgisser \cite[Theorem 3.5]{AB} for a certain natural distribution on the data $B_j$ and $C$.

The projective dual to the irreducible variety of rank $n-r$ matrices in $Q$ is
 a hypersurface of degree $\delta(p+1,n,r)$, by \cite[Theorem 13]{NieRanStu}.
A general point in this hypersurface is a matrix of rank $r$.
The set of these hypersurfaces, as $r$ ranges over \eqref{eq:patakirange}, contains
the irreducible components $X$ of the algebraic boundary of $Q^o$.

Because $\partial S\subset \partial Q^\circ$, we get the inclusion
\begin{equation}
\label{eq:cutgeneric}
\partial_{\rm alg} (S) \,\,\,\subset\,\, \bigcup_{X\subset \partial_{\rm alg} (Q^o)} \!\! X\,\cap\, \lspan(A_1,\ldots,A_d),
\end{equation}
where $X$ runs over all irreducible components of the algebraic boundary of $Q^o$.
Since the symmetric matrices $A_1,\ldots, A_d$ are generic, the varieties $X\cap\lspan(A_1,\ldots,A_d)$ are irreducible hypersurfaces of degree $\delta(p+1,n,r)$ by Bertini's Theorem.
The irreducible components of $\partial_{\rm alg}(S)$ are among these and each one has a positive probability of occurring because the semi-algebraic set $X\cap \partial Q^\circ$ is open in $\partial Q^\circ$.
\end{proof}

\begin{Rem}\label{Rem:Zariskiclos}
Here is an important technical point regarding Theorem \ref{Thm:algBound}.
Let $K$ be the subfield of $\R$ generated by the entries of $A_i,B_j,C$.
Typically, we will have $K = \mathbb{Q}$.

The irreducibility statement in Theorem \ref{Thm:algBound} is understood
with respect to the ground field $K$. The rank $r$ stratum in $\partial_{\rm alg}(S)$
is irreducible as a variety over $K$, but it may break into irreducible
components over $\R$. Consider the case $d=1$: here $S$ is a line segment,
so its algebraic boundary over $\R$ consists of two points, but its algebraic
boundary over $K$ is an irreducible $0$-dimensional scheme of length
$\delta(p+1,n,r)$. That algebraic statement is the whole point of \cite{NieRanStu}.
For spectrahedral shadows $S$ of higher dimension $d \geq 2$, 
the $K$-components of $\partial_{\rm alg}(S)$ usually
coincide with the $\R$-components. But there are some notable exceptions,
as we shall see in Section 3.
\end{Rem}

We next present a table of values for the degrees $\delta(p+1,n,r)$
of the boundary hypersurfaces in Theorem~\ref{Thm:algBound}.
This table extends the one in \cite{NieRanStu}
and it can be computed using the formula in \cite{BotRan}.
The rows are indexed by the codimension $p$ of the projection,
and the columns are indexed by the matrix size $n$.
Each box contains all ranks $r$ in the Pataki range (\ref{eq:patakirange}).
For example, consider the three entries
for $p = 5$ and $n=6$. If $d=2$ then $S$ is a planar spectrahedral
shadow. It is the projection of a 
$7$-dimensional spectrahedron of $6 {\times} 6$-matrices.
The points in the boundary of $S$ have rank $5$, $4$, or $3$.
The corresponding irreducible plane curves have degrees 
$32$, $1400$ and $112$ respectively.

The first row $p=1$ concerns codimension $1$ projections.
We start with a $(d+1)$-dimensional spectrahedron
of $n \times n$-matrices and we project it into $\R^d$.
The ramification locus of the projection is defined by
the determinant (of degree $n$) and one of its directional derivatives (of degree $n-1$). By B\'ezout's
Theorem, the ramification locus has degree $n(n-1)$,
and hence so does the branch locus. This explains the formula 
$\delta(2,n,1) = n(n-1)$ for the entries in the first row of the table.

{\small
\begin{table}
\begin{center}
\begin{tabular}{|c|cc|cc|cc|cc|cc|cc|cc|cc|} \hline
& \multicolumn{2}{c|}{$n=3$} & \multicolumn{2}{c|}{$n=4$} & \multicolumn{2}{c|}{$n=5$} & \multicolumn{2}{c|}{$n=6$} & \multicolumn{2}{c|}{$n=7$} &
\multicolumn{2}{c|}{$n=8$} & \multicolumn{2}{c|}{$n=9$} & \multicolumn{2}{c|}{$n=10$}\\
\hline
$p$ & $r$ & ${\rm deg}$ & $r$ & ${\rm deg}$ & $r$ & ${\rm deg}$
& $r$ & ${\rm deg}$ & $r$ & ${\rm deg}$ & $r$ & ${\rm deg}$ & $r$ & ${\rm deg}$ & $r$ & ${\rm deg}$ \\
\hline
1 & 2 & 6 & 3 & 12 & 4 & 20 & 5 & 30  & 6 & 42 & 7 & 56 & 8 & 72 & 9 & 90 \\
\hline
2 &  2 & 4 & 3 & 16 & 4 & 40 & 5 & 80 & 6 & 140 & 7 & 224 & 8 & 336 & 9 & 480 \\
  &  1 & 4 & 2 & 10 & 3 & 20 & 4 & 35 & 5 &  56 & 6 &  84 & 7 & 120 & 8 & 165 \\
\hline
3 & 1 & 6 & 3 & 8  & 4 & 40 & 5 & 120 & 6 & 280 & 7 & 560 & 8 & 1008 & 9 & 1680 \\
  &   &   & 2 & 30 & 3 & 90 & 4 & 210 & 5 & 420 & 6 & 756 & 7 & 1260 & 8 & 1980 \\
\hline
4 & 1 & 3 & 2 & 42 & 4 & 16  & 5 & 96  & 6 & 336  & 7 & 896  & 8 & 2016 & 9 & 4032 \\
  & & &   &    & 3 & 207 & 4 & 672 & 5 & 1722 & 6 & 3780 & 7 & 7434 & 8 & 13464 \\
\hline
5 & & & 2 & 30 & 3 & 290 & 5 & 32   & 6 & 224 & 7 & 896 & 8 & 2688 & 9 & 6720\\
  & & & 1 & 8  & 2 & 35  & 4 & 1400 & 5 & 4760& 6 & 13020& 7 & 30660& 8 & 64680 \\
  & & &   &    &   &     & 3 & 112  & 4 & 294 & 5 & 672 & 6 & 1386 & 7 & 2640 \\
\hline
6 & & & 2 & 10 & 3 & 260 & 4 & 2040 & 6 & 64   & 7 & 512  & 8 & 2304 & 9 & 7680 \\
  & & & 1 & 16 & 2 & 140 & 3 & 672  & 5 & 9600 & 6 & 33540& 7 & 96120& 8 & 238920\\
  & & &   &    &   &     &   &      & 4 & 2352 & 5 & 6720 & 6 & 16632 & 7 & 36960 \\
\hline
7 & & & 1 & 12 & 3 & 140 & 4 & 2100 & 5 & 14532& 7 &    128 & 8 & 1152 & 9 & 5760 \\
  & & &   &    & 2 & 260 & 3 & 1992 & 4 & 9576 & 6 &  66948 & 7 & 238140 & 8 & 706860 \\
  & & &   &    &   &     &   &      &   &      & 5 & 34800 & 6 & 104544 & 7 & 273240 \\
\hline
8 & & & 1 & 4 & 3 & 35  & 4 & 1470 & 5 & 16485& 6 & 104692 & 8 & 256 & 9 & 2560 \\
  & & & & & 2 & 290 & 3 & 3812 & 4 & 25998 & 5 & 122400 & 7 & 474145 & 8 &1708630 \\
  & & & & &   &     &   &      &   &        &   &         & 6 & 451638 & 7 & 1399860\\
\hline
9 & &  & & & 2 & 207 & 4 & 630  & 5 & 13650 & 6 & 127596& 7 & 761364& 9 & 512 \\
  & &  & & & 1 &  16 & 3 & 5184 & 4 & 52143& 5 & 324624& 6 & 1490049& 8 & 3401574\\
  & &  & & &   &     & 2 & 126  & 3 & 672& 4 & 2772 & 5 & 9504& 7 & 5524728\\
  & &  & & &   &     &   & &   & &   & &   & & 6 & 28314\\
\hline
\end{tabular}
\end{center}
\smallskip
\caption{Degrees of the
boundary components of generic spectrahedral shadows}
\label{degtab}
\end{table}
}

The case $p=0$ is omitted from Table \ref{degtab} because
a spectrahedral shadow of type $(n,d,0)$ is just
a $d$-dimensional spectrahedron of $n \times n$-matrices.
The algebraic boundary of such a spectrahedron is
a hypersurface of degree $n$, given by the determinant.

\begin{Rem}
In the next sections we compute the equations defining $\partial_{\rm alg}(S)$
for various examples.  This was done using the software {\tt Macaulay2} \cite{M2}. 
In each case, we start with a determinantal ideal  of
a symmetric matrix of linear forms. This represents a particular rank stratum in the
upstairs spectrahedron. We then compute the ramification locus of the projection $\pi$,
and we finally take the image of this ramification to obtain the branch locus downstairs.
Such elimination problems are generally hard.
\end{Rem}

\smallskip

We close this section by recording the following special case from the proof above.

\begin{Cor} 
\label{cor:NoSlice}
If $p = \binom{n+1}{2}-d-1$ then the generic spectrahedral shadow $S$ is the convex body dual to 
a generic $d$-dimensional spectrahedron $Q$ of $n \times n$-matrices.
\end{Cor}

\section{A gallery of shadows}

In this section we discuss examples of
spectrahedral shadows. Most of these have dimension $d=2$.
We start with projections from $3$-space into the plane, so that $p=1$.

\begin{Exm}
\label{Exm:321} 
A spectrahedral shadow $S$ of type $(3,2,1)$ is 
the projection into $\R^2$
of a cubic spectrahedron in $\R^3$. An example of
such a spectrahedron is the {\em elliptope}
\begin{equation}
\label{eq:elliptope}  \mathcal{E}_3  \quad = \quad
\biggl\{ \,(x,y,z) \in \R^3\,:\,
\left(
\begin{array}[]{ccc}
1 & x & y \\
x & 1 & z \\
y & z & 1
\end{array}\right) \succeq 0 \, \biggr\}.
\end{equation}
If the projection is generic then $S$
is a planar convex body whose algebraic boundary
$\partial_{\rm alg}(S)$ is a rational curve of degree $6$. 
Indeed, the ramification of the projection is the curve
in $\R^3$ defined by a cubic and one of its directional
derivatives, which is a quadric. The curve has degree $6$. Its expected genus would be 
$4$, cf.~\cite[Remark 6.4.1(b)]{Har}. But it has $4$ singular points coming from the $4$ nodes of the cubic
surface, cf.~\cite{Quartic}. So, the curve is rational.
 Its projection into the plane is a sextic curve with $6$ singularities coming from the projection and $4$ singularities coming from the nodes of the cubic. 
\end{Exm}

Moving up by one degree, we now come to 
quartic spectrahedra in $\R^3$. These were recently studied by
Ottem {\it et al.}~in \cite{Quartic}. We consider their projections into the plane.

\begin{Exm}\label{Exm:421}
We now consider type $(4,2,1)$. According to \cite{Quartic}, there are $20$
different semi-algebraically generic types of quartic spectrahedra in $\R^3$.
They are classified according to their number of nodal singularities on
and off the boundary surface. The total number of complex nodes is $10$.
We consider a generic projection into the plane. Now, the ramification locus
is a complete intersection of the quartic and a cubic. This is a 
 curve of degree $12$ with $10$ singular points, namely the nodes of the quartic surface.
The genus of such a curve is $9$;
 cf.~\cite[Remark 6.4.1(b)]{Har}. The projection into the plane is a curve of degree $12$
 of genus $9$, so it has $55 - 9 = 46 $ singular points.
Figure \ref{fig:421} shows an example where all $10$ nodes are on the spectrahedron.
 Their images in the plane are the $10$ blue points. In addition, this curve of degree $12$
 has $36$ singularities that are double points of the projection. These are marked in red in
Figure \ref{fig:421}.

\begin{figure}[h]
\includegraphics{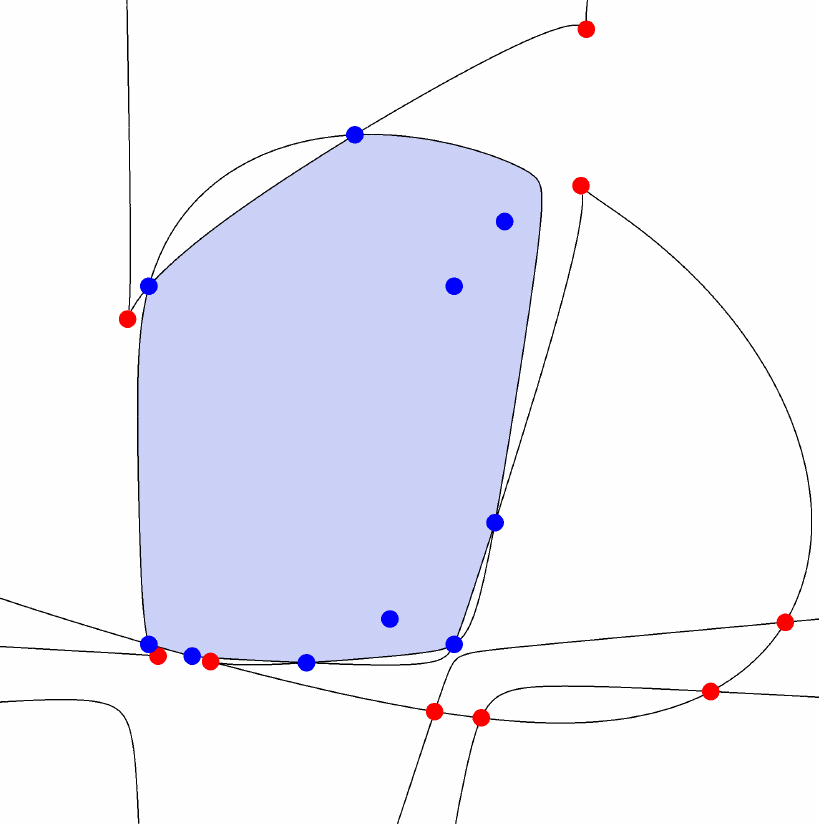}
\caption{A generic planar spectrahedral shadow of type $(4,2,1)$ is bounded by a curve of degree $12$
that has $46 = 10+36$ singular points.}
\label{fig:421}
\end{figure}
\end{Exm}

We now come to the second row ($p=2$) of Table \ref{degtab},
where the projection is from $4$-space into the plane.
Already the first entry $n=3$ is quite interesting and beautiful.

\begin{figure}[h]
\centering
\includegraphics[scale=1]{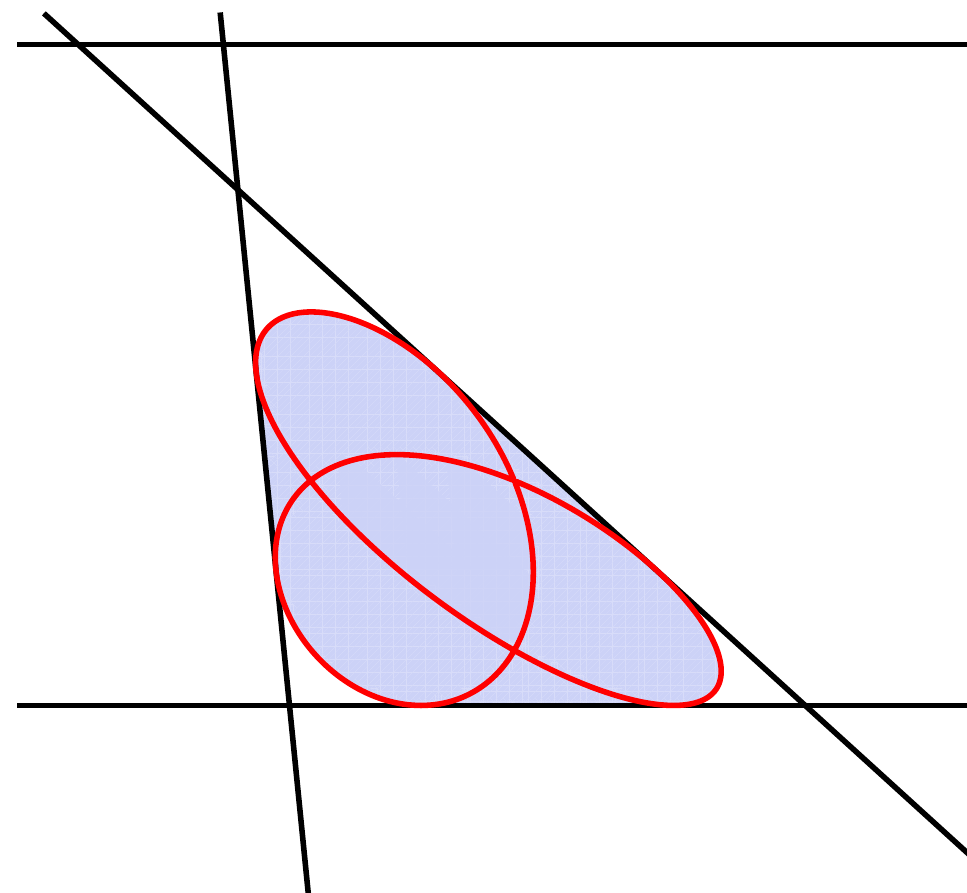}
\vspace{-0.1in}
\caption{The shaded area is a spectrahedral shadow of type $(3,2,2)$.}
\label{fig:322}
\end{figure}

\begin{Exm}\label{Exm:322}
Figure \ref{fig:322} shows a planar spectrahedral shadow of type $(3,2,2)$.
This is the projection of the following $4$-dimensional spectrahedron into the
$(x_1,x_2)$-plane:
\[
\mathcal{A} \,\,:=\,\, \left(
\begin{array}[]{ccc}
y_1 & x_1 & x_2 \\
x_1 & y_2 & -x_1-\frac{6}{5}x_2-\frac{7}{10} \\
x_2 & -x_1-\frac{6}{5}x_2-\frac{7}{10} & 2-y_1-y_2
\end{array}
\right)  \,\, \succeq \,\, 0.
\]
The points of rank $1$ on the spectrahedron define an irreducible quartic curve in $\R^4$,
and the red curve in Figure \ref{fig:322} is the projection of that curve
into $\R^2$. It is a quartic of genus $0$, because it has $3$ double points. We obtain its equation by
eliminating $y_1$ and $y_2$ from the $2 \times 2$-minors of $\mathcal{A}$:
\begin{equation}
\label{eq:quartic1}
 \begin{matrix}
100 x_1^{4}+240 x_1^{3} x_2+344 x_1^{2} x_2^{2}+240 x_1 x_2^{3}+144 x_2^{4}
+140 x_1^{3}+368 x_1^{2} x_2 \\
+ 380 x_1 x_2^{2}+168 x_2^{3}+49 x_1^{2}+140 x_1 x_2+49 x_2^{2}.
\end{matrix}
\end{equation}
The four black lines in Figure~\ref{fig:322} represent the rank $2$ locus on the
spectrahedral shadow. Their union is the reducible quartic curve whose defining polynomial equals
\begin{equation}
\label{eq:quartic2}
(2 x_2-3) (22 x_2+17) (20 x_1+2 x_2+17) (20 x_1+22 x_2-3).
\end{equation}
This quartic is obtained from $f = {\rm det}(\mathcal{A})$ by 
eliminating $y_1$ and $y_2$ from the saturation of the ideal
$\bigl\langle f, \,\partial f/\partial y_1,\,\partial f/\partial y_2 \bigr\rangle$
with respect to the $2 \times 2$-minors of $\mathcal{A}$.
The entries $4$ and $4$ in the box for $p=2,n=3$ in Table \ref{degtab}
correspond to the quartics in (\ref{eq:quartic1}) and~(\ref{eq:quartic2}).

It is instructive to examine the proof of Theorem \ref{Thm:algBound}
for this example. The data are
{\small
\[
A_1 = 
\begin{bmatrix}
\,0 & \!\!\! 1 & 0 \\
\, 1 & \! 0 & \!\!\! -1 \\
\,0 & \!\! -1 & 0
\end{bmatrix}\!\!,\,
A_2 = \begin{bmatrix}
0 & \! 0 & \!\! 1 \\
0 & \! 0 & \!\!\!\! -\frac{6}{5} \\
\! 1 & \!\!\! -\frac{6}{5} & \! 0 
\end{bmatrix}\!\!,\,
B_1 = 
\begin{bmatrix}
1 & \! 0 \! & 0 \\
0 & \! 0 \! & 0 \\
0 & \! 0 \! & \!\! \! -1
\end{bmatrix}\!\!,\,
B_2 =  \begin{bmatrix}
0 & \! 0 \! & 0 \\
0 & \! 1 \! & 0 \\
0 & \! 0 \! & \!\!\! -1
\end{bmatrix}\!\!,\,
C = \begin{bmatrix}
0 & \!0\! & 0 \\
0 & \! 0 \! & \!\! \! -\frac{7}{10} \\
0 & \!\! -\frac{7}{10}\! & 2
\end{bmatrix} \! .
\]
}
The following spectrahedron, seen in (\ref{eq:QQQ}), is isomorphic to the elliptope 
$\mathcal{E}_3$ in (\ref{eq:elliptope}):
\[
Q \,\,=\,\, \bigl\{U\in \mathcal{S}_3^+\colon \langle B_1,U\rangle =0, \langle B_2,U\rangle =0, 
\langle C,U\rangle =1 \bigr\}.
\]
The convex body $Q^o$ dual to $Q$
is bounded by the quartic Steiner surface and four planes,
as seen in \cite[Figure 4]{StUh}. Our spectrahedral shadow $S$ is the
two-dimensional section $ \,Q^o \,\cap\, \lspan(A_1,A_2)$.
The rational quartic in Figure \ref{fig:322} is the section of the Steiner surface,
and the lines arise from the four planes. One of the three lines is
not part of the algebraic boundary for our particular $S$.
When the data $A_1,A_2,B_1,B_2,C$ are more general, then the four lines
form a quartic curve that is irreducible over $\Q$. So, over $\Q$, the algebraic boundary
$\partial_{\rm alg}(S)$ has degree $4+4$, as predicted by Theorem~\ref{Thm:algBound}.
\end{Exm}

In Examples \ref{Exm:321} and \ref{Exm:322} we discussed
planar spectrahedral shadows that are defined by $3 \times 3$-matrices.
The last relevant case in this family occurs for $p=3$.

\begin{Exm}
This is an illustration of Corollary \ref{cor:NoSlice} for $n=3$ and $d=2$.
A spectrahedral shadow $S$ of type $(3,2,3)$ is the projection of 
a $5$-dimensional spectrahedron into $\R^2$. Alternatively,
$S = Q^o$ is the dual to a cubic spectrahedron $Q$ in $\R^2$.
Since $\partial_{\rm alg}(Q)$ is a smooth cubic curve, its dual
$\partial_{\rm alg}(S)$ is a curve of genus $1$ and degree $6$.
It has $9$ cusps and no ordinary double points, by Pl\"ucker's formulas for plane curves.
\end{Exm}

Here is an analogous example in dimension $d=3$.

\begin{Exm}
Let $Q$ be a generic quartic spectrahedron in $\R^3$, as in \cite{Quartic}.
Its convex dual $S = Q^o$ is a spectrahedral shadow of type $(4,3,6)$,
so we obtain it as a projection from $\R^9$ to $\R^3$.
The algebraic boundary $\partial_{\rm alg}(S)$ has
two components, one for rank $1$ and one for rank $2$. The first is
the surface of degree $16$ that is dual to the quartic $\partial_{\rm alg}(Q)$.
The second is the arrangement of $10$ planes dual to the $10$ singular points in
$\partial_{\rm alg}(Q)$.
The latter component factors over the reals whereas it is generically irreducible over~$\Q$. 
\end{Exm}

Our final example concerns the lower right entry in 
Table \ref{degtab}.

\begin{Exm}
The smallest case when the Pataki range consists of four possible ranks
occurs for $p=9$ and $n=10$. Let us focus on plane curves, so we 
consider generic spectrahedral shadows $S$ of type $(10,2,9)$.
Here, the algebraic boundary $\partial_{\rm alg}(S)$ may have up to four
components, representing the ranks $6,7,8$ and $9$.
These components are irreducible plane curves of degrees
$28314$, $5524728$, $3401574$ and $512$ respectively.
\end{Exm}

\section{Perturbations and degenerations}

Theorem \ref{Thm:algBound} characterizes the algebraic
boundaries of spectrahedral shadows whose defining
matrices $A_i, B_j, C$ are generic. This genericity assumption 
stands in contrast to what researchers in 
optimization actually care about. Indeed, when modeling convex sets with
 semidefinite representations, the matrices $A_i,B_j,C$ must be chosen
to have a very special structure.
So, the question is how to make our algebraic geometry results
relevant for applications of semidefinite programming, such as those discussed in \cite{BPT}.
That question will require a further research effort that goes beyond the present paper.

Geometrically, we need to study the effects of perturbations, from special data to generic data,
and of degenerations, from generic data to special data. 
Our final section is meant to illustrate some of the issues that need to be addressed.
We focus on three simple examples, where the spectrahedral shadow $S$
is a convex polygon in the plane.
\begin{Exm}
Consider the special quartic spectrahedron in $\R^3$ that is defined by
\[
\mathcal{A} \,\,=\,\, \left(
\begin{array}[]{cccc}
1   & x_1 & x_2 & y \\
x_1 & 1   & y   & x_2 \\
x_2 & y   & 1   & x_1 \\
y   & x_2 & x_1 & 1
\end{array}\right) \,\, \succeq \,\, 0.
\]
This is the tetrahedron with vertices $(1,1,1), (1,-1,-1), (-1,1,-1)$ and $(-1,-1,1)$.
Its projection into the $(x_1,x_2)$-plane is the square $S$ with vertices $(\pm 1,\pm 1)$.
This realizes $S$ as a spectrahedral shadow of type $(4,2,1)$. 
The algebraic boundary $\partial_{\rm alg}(S)$ consists of the four boundary lines
of the square $S$. There are two further lines, connecting antipodal vertices of
the square, where the map also branches. Indeed, if we
eliminate $y$ from $\bigl\langle {\rm det}(\mathcal{A}),\,\partial {\rm det}(\mathcal{A})/\partial y
\bigr\rangle$, then we get this polynomial of degree $12$:
\begin{equation}
\label{eq:sixfactors} (x_1-1)^2 (x_1+1)^2 (x_2-1)^2 (x_2+1)^2 (x_1-x_2)^2 (x_1+x_2)^2 .
\end{equation}
Thus, algebraically, this computation is consistent with Example \ref{Exm:421},
which tells us that, for generic data, one expects the branch
curve to have degree $12$ and genus $9$.

The ramification locus of the rank $2$-variety projects onto the $4$ vertices of $S$.

\begin{figure}[h]
\begin{center}
\begin{minipage}[]{0.4\textwidth}
\begin{center}
\includegraphics[scale = 0.7]{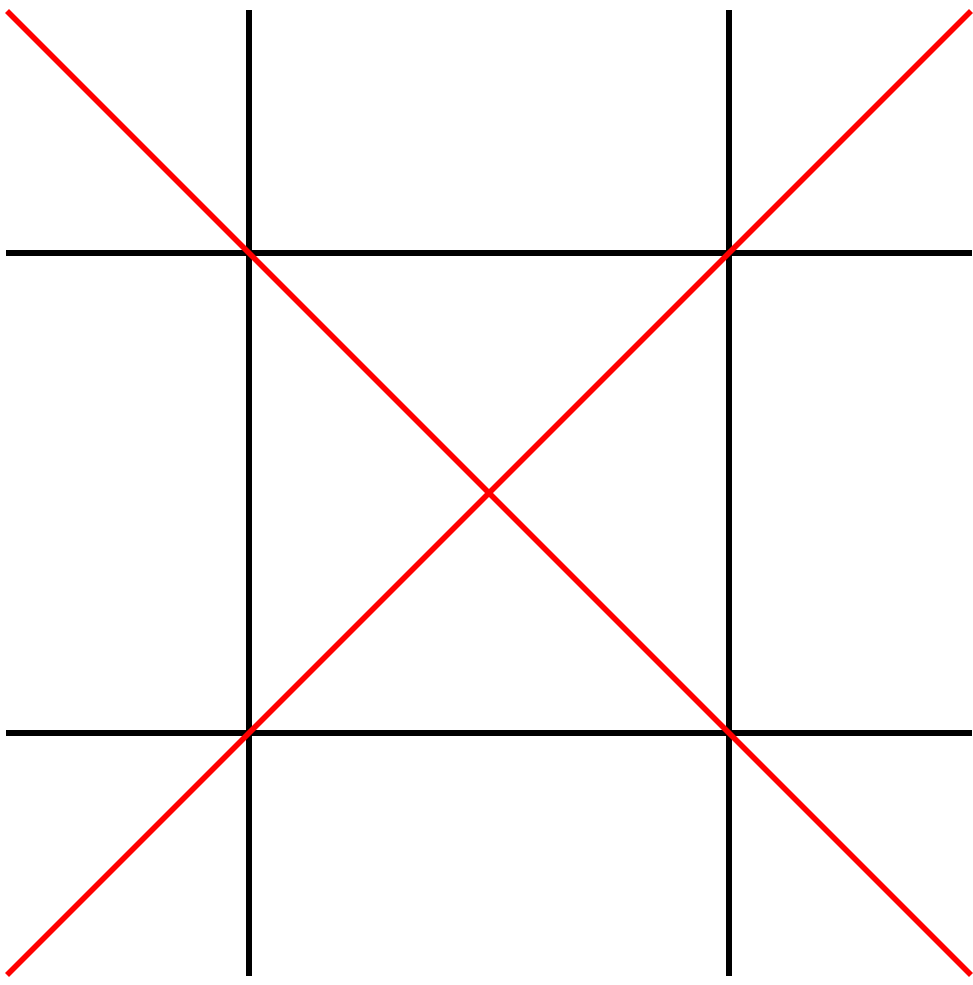}
\end{center}
\end{minipage} \qquad
\begin{minipage}[]{0.4\textwidth}
\begin{center}
\includegraphics[scale = 0.7]{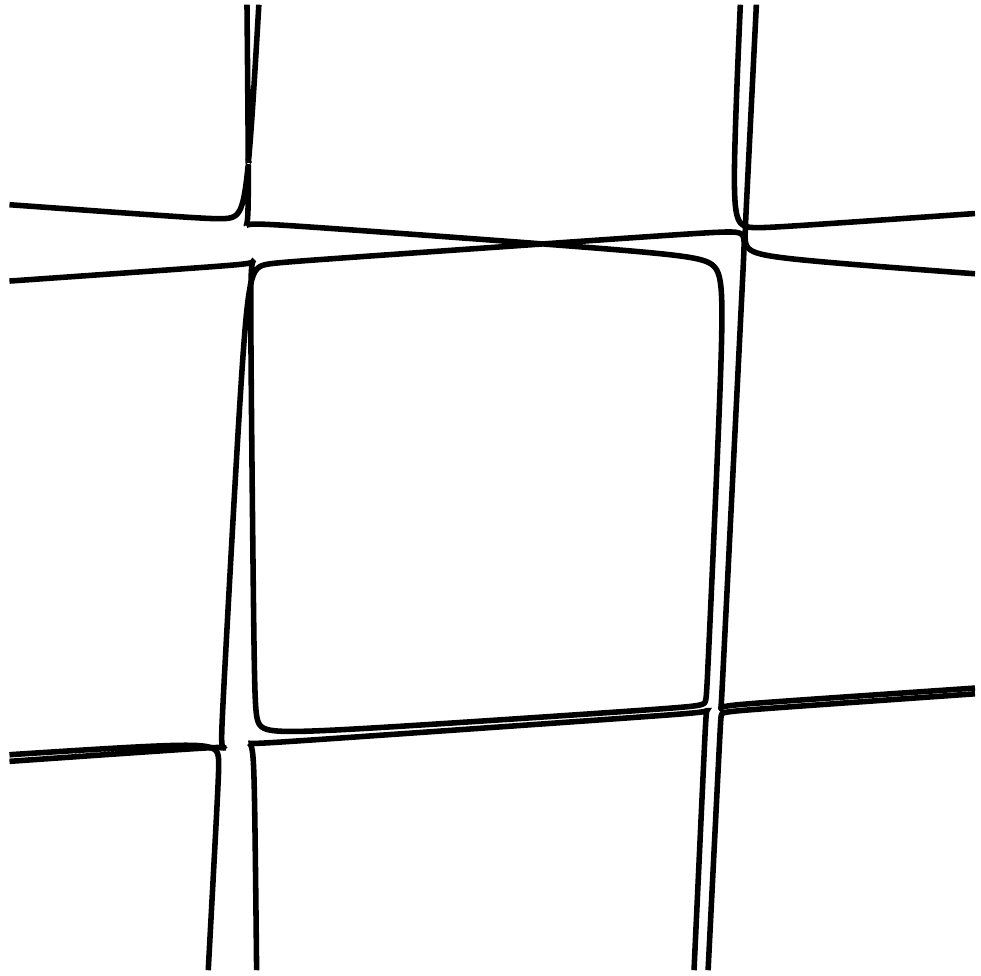}
\end{center}
\end{minipage}
\end{center}
\caption{The square on the left and a generic perturbation on the right.}
\label{fig:square}
\end{figure}

We now perturb the representation of $S$ by replacing 
the given matrix $\mathcal{A}$ with
\[
\mathcal{A}_\epsilon \quad = \quad
\left(
\begin{array}[]{cccc}
1   & x_1 & x_2 & y \\
x_1 & 1   & y   & x_2 \\
x_2 & y   & 1   & x_1 \\
y   & x_2 & x_1 & 1
\end{array}\right) \,+\,\,\epsilon
\left(
\begin{array}[]{cccc}
2y & 0  & 0  & 0 \\
0  & 3x_2 & 0  & 0 \\
0  & 0  & 5y & 0 \\
0  & 0  & 0  & -7x_1
\end{array}\right).
\]
The resulting spectrahedral shadow $S_\epsilon$ of type $(4,2,1)$ is generic
for almost all real choices of $\epsilon$.
Its algebraic boundary $\partial_{\rm alg}(S_\epsilon)$ is defined by
an irreducible polynomial of degree $12$ that specializes to
the perfect square in (\ref{eq:sixfactors}) when $\epsilon \rightarrow 0$.
For small $\epsilon > 0$,
each of the two red double lines becomes a complex conjugate pair,
so it disappears from the picture in $\R^2$.
The curve $\partial_{\rm alg}(S_\epsilon)$ is shown for
$\epsilon=\frac{1}{50}$ on the right in 
Figure~\ref{fig:square}.
\end{Exm}

We next discuss representations of
hexagons as spectrahedral shadows with $n=4$.

\begin{Exm}
Let $S$ be the hexagon in $\R^2$ with vertices $(-1,1)$, $(-1,0)$, $(0,-1)$, $(1,-1)$, $(1,0)$, and $(0,1)$. It is the spectrahedral shadow of type $(4,2,3)$ given by
\[
\left(
\begin{array}[]{cccc}
1 & x_1 & x_2 & x_1+x_2 \\
x_1 & 1 & y_1 & y_2   \\
x_2 & y_1 & 1 & y_3   \\
x_1+x_2 & y_2 & y_3 & 1 
\end{array}\right) \,\,\succeq \,\,0.
\]
This representation is due to Jo\~ ao Gouveia.
Each point on an edge of the hexagon $S$ has rank $3$. The inverse image of each edge
is a $2$-dimensional face of the $5$-dimensional spectrahedron upstairs. That face is 
a planar spectrahedron of degree $3$.
For instance, over the edge between $(-1,1)$ and $(0,1)$, that cubic spectrahedron is defined by
\[
\hbox{$x_2=1$, $\,x_1=y_1$, $\,y_3=x_1+1$, \ \ and} \qquad
\left(
\begin{array}[]{ccc}
1 & y_1 & y_1+1 \\
y_1 & 1 & y_2 \\
y_1+1 & y_2 & 1
\end{array}\right) \,\,\succeq \,\, 0.
\]
The boundary of this face consists of matrices of rank $2$,
and it maps onto the edge of $S$. Thus, in this degenerate example, the
 branch loci for rank $2$ and rank $3$ coincide.

Compare this with what happens generically for type $(4,2,3)$.
According to Table~\ref{degtab}, the rank $3$ boundary is a curve
of degree $8$, and the rank $2$ boundary is a curve of degree $30$.
This generic behavior is realized by the following explicit perturbation
\[
\mathcal{H}_\epsilon = 
\begin{pmatrix}
1 & x_1 & x_2 & x_1+x_2 \\
x_1 & 1 & y_1 & y_2   \\
x_2 & y_1 & 1 & y_3   \\
x_1+x_2 \! & y_2 & y_3 & 1 
\end{pmatrix}
 + \epsilon 
 \begin{pmatrix}
 7y_1 & 2y_2 & 3y_3 & 0 \\
2y_2 & 5y_3 & 0 & 0 \\
3y_3 & 0 & 3y_1 & 0 \\
0 & 0 & 0 & 5y_2 +y_3
\end{pmatrix}.
\]
For small values $\epsilon > 0$, we find that the branch loci
are irreducible curves of degree $8$ and $30$. These curves
collapse to the six lines when $\epsilon \rightarrow 0$, but
the situation is apparently more complicated than the
nice family $\mathcal{A}_\epsilon$ seen in the previous example. 
The polynomials of degree $8$ and $30$ defining this branch locus can be 
computed purely symbolically in {\tt Macaulay2} \cite{M2},
by way of specialization, elimination and interpolation. 
\end{Exm}

\begin{Exm}\label{Exm:pablo}
Let $S$ denote the hexagon in $\R^2$ with vertices $(\pm \frac12,\pm \sqrt\frac34)$ and $(\pm 1,0)$.
The following is a representation of $S$ as a spectrahedral shadow of type $(4,2,3)$:
\begin{equation}
\label{ex:HamJam}
\left(
\begin{array}[]{cccc}
1 & x_1 & x_2 & y_3 \\
x_1 & \frac12(1+y_1) & \frac12 y_2 & y_1 \\
x_2 & \frac12 y_2 & \frac12 (1-y_1) & -y_2 \\
y_3 & y_1 & -y_2 & 1
\end{array}\right)\,\, \succeq \,\,0.
\end{equation}
This matrix is due to Hamza Fawzi and James Saunderson.

\begin{figure}[h]
\includegraphics[scale = 0.8]{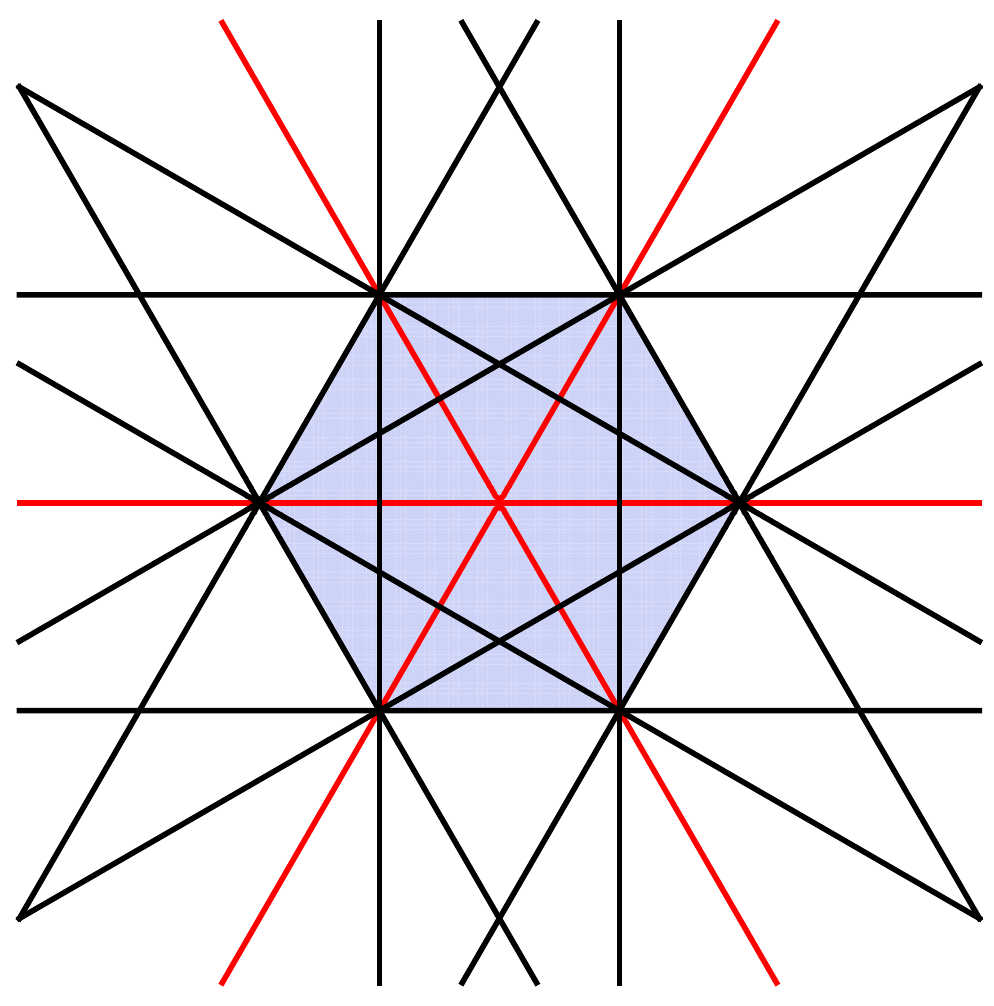}
\vspace{-0.1in}
\caption{A curve of degree $30$ degenerates to the $15$ lines in a hexagon $S$.}
\label{fig:pablo}
\end{figure}

The projection is ramified only along points of rank $2$, {\it i.e.}~there is no smooth point of rank $3$ 
whose tangent space to the determinantal hypersurface contains the kernel of the projection.
The branch locus of rank $2$ is a non-reduced curve of degree $18$:
\begin{equation}\label{eqn:pablohex}
\begin{matrix}
 y^2 (3 x^2-y^2)^2(2 x-1) (2 x+1) (4 y^2-3) (x^2-3 y^2-2 x+1)\cdot & \\
  (x^2-3 y^2+2 x+1) (3 x^2-y^2-6 x+3) (3 x^2-y^2+6x+3) & =\,\,  0.
  \end{matrix}
\end{equation}
This equation factors into linear terms over $\Q[\sqrt{3}]$. The three lines of multiplicity $2$ in this curve are the three lines through the origin.
They are spanned by antipodal pairs of vertices of the hexagon $S$.
The hexagon and its $15$ lines are shown in Figure \ref{fig:pablo}.

A perturbation into general position can be obtained by adding $\epsilon \diag(5y_2,7y_3,0,3x_1)$ to 
the matrix in (\ref{ex:HamJam}). For small $\epsilon > 0$, the algebraic boundary of the
spectrahedral shadow now consists of two irreducible
curves, having degree $8$ and $30$, as predicted in Table \ref{degtab}.
As $\epsilon$ approaches $0$, the polynomial of degree $30$
converges to the polynomial (\ref{eqn:pablohex}) of degree $18$, as expected. 
In the limit, the line at infinity has multiplicity $12$.
\end{Exm}

\bigskip

\noindent \textbf{Acknowledgements.} We thank Pablo Parrilo and James Saunderson for the discussions that started this project, and Kristian Ranestad and Cynthia Vinzant for helpful comments. 
This work was done at 
the National Institute of Mathematical Sciences,
Daejeon, Korea, during the Summer 2014 Program on {\em Applied Algebraic Geometry}.

\bigskip
\medskip

\noindent
\footnotesize {\bf Authors' addresses:}

\noindent Rainer Sinn, Georgia Institute of Technology,
Atlanta, GA 30332, USA,
{\tt rainer.sinn@uni-konstanz.de}

\noindent Bernd Sturmfels, University of California, Berkeley, CA 94720, USA,
{\tt bernd@math.berkeley.edu}
\end{document}